\documentclass{amsart}
\usepackage[dvips]{graphicx}
\usepackage{enumerate}

\usepackage{amssymb}
\usepackage{amscd}
\usepackage{amscd}
\usepackage{color}
\usepackage{here}

\begin{document}
\title[Observable Lyapunov irregular sets]{observable Lyapunov irregular sets for 
planar piecewise expanding maps}

\author{Yushi Nakano}
\address[Yushi Nakano]{Department of Mathematics, Tokai University, 4-1-1 Kitakaname, Hiratuka, Kanagawa, 259-1292, JAPAN}
\email{yushi.nakano@tsc.u-tokai.ac.jp}

\author{Teruhiko Soma}
\address[Teruhiko Soma]{Department of Mathematical Sciences, Tokyo Metropolitan University, 1-1 Minami-Ohsawa, Hachioji, Tokyo, 192-0397, JAPAN}
\email{tsoma@tmu.ac.jp}

\author{Kodai Yamamoto}
\address[Kodai Yamamoto]{Joint Graduate School of Mathematics for Innovation, 
Kyushu University, Fukuoka, 819-0395, JAPAN}
\email{yamamoto.kodai.508@s.kyushu-u.ac.jp}

\date{\today}

\begin{abstract}
For any integer $r$ with $1\leq r<\infty$, 
we present a one-parameter family $F_\sigma$ $(0<\sigma<1)$ of 2-dimensional piecewise $\mathcal C^r$ expanding maps 
such that each $F_\sigma$ has an observable (i.e.\ Lebesgue positive) Lyapunov irregular set.
These maps are obtained by modifying the piecewise expanding map given in Tsujii \cite{tj}.
In strong contrast to it, we  also show that any Lyapunov irregular set of any 2-dimensional piecewise real analytic 
expanding map is not observable.
This is based on the spectral analysis of piecewise  expanding maps in 
Buzzi \cite{bu}.
\end{abstract}

\maketitle

\newtheorem{theorem}{Theorem}[section]
\newtheorem{cor}[theorem]{Corollary}
\newtheorem{lemma}[theorem]{Lemma}
\newtheorem{prop}[theorem]{Proposition}

\newtheorem{mtheorem}{Theorem}
\renewcommand{\themtheorem}{\Alph{mtheorem}}
\newtheorem{mcorollary}[mtheorem]{Corollary}

\theoremstyle{definition}
\newtheorem{definition}[theorem]{Definition}
\newtheorem{example}[theorem]{Example}
\newtheorem{remark}[theorem]{Remark}
\newtheorem{question}[theorem]{Question}

\numberwithin{figure}{section}
\numberwithin{equation}{section}

\def\setl{\setlength{\leftskip}{-18pt}}

\section{Introduction}

Lyapunov exponents are one of important measuring tools in the study of chaotic phenomena from 
mathematical or physical points of view.
In general, it is not easy to calculate the exact value of Lyapunov exponent theoretically.
In practical examples, we usually rely on numerical calculations, but 
we know that it has various technical difficulties.
See \cite[Section 12.5]{pjs} for example.

Let $M$ be a Riemannian manifold and $f:M\longrightarrow M$ a differentiable map such that the support $\mathrm{supp}(f)$ of $f$ is compact.
For a given $x\in M$ and non-zero $\boldsymbol{v}\in T_xM$, we study 
the $\omega$-limit set $\omega(\mathcal{A}_f(x,{\boldsymbol{v}}))$ of the sequence 
\begin{equation}\label{eqn_Af}
\mathcal{A}_f(x,{\boldsymbol{v}})=\left\{ \dfrac1{n}\log\|Df^n(x){\boldsymbol{v}}\|\ \Big|\ n=1,2,\dots\right\}.
\end{equation}
Here, for a sequence $A=\{a_n\}_{n=1}^\infty$ in $\mathbb{R}$, 
the $\omega$-\emph{limit set} $\omega(A)$ of $A$ is the set of points $a$ 
such that, for any $\varepsilon>0$ and any positive integer $n_0$, there exists $n>n_0$ 
with $(a-\varepsilon,a+\varepsilon)\ni a_n$.

If $M$ has an $f$-invariant 
 measure $\mu$ on $\mathrm{supp}(f)$, 
then Oseledets' Multiplicative Ergodic Theorem \cite{os} implies that, for $\mu$-almost every $x\in M$ and 
every non-zero ${\boldsymbol{v}}\in T_xM$, the limit of $\mathcal{A}_f(x,{\boldsymbol{v}})$ exists or equivalently the $\omega$-limit set $\omega(\mathcal{A}_f(x,{\boldsymbol{v}}))$ is a single point set, 
see \cite[Chapter 4]{vi} for details.
The limit is called the \emph{Lyapunov exponent} of $f$ at $x$ with respect to ${\boldsymbol{v}}$.

We are interested in the case when such a limit does not exist.
An element $x$ of $M$ is said to be a \emph{Lyapunov irregular point} of $f$ if 
$\omega(\mathcal{A}_f(x,{\boldsymbol{v}}))$ is not a single point for some non-zero ${\boldsymbol{v}}\in T_xM$.
A subset of $M$ consisting of Lyapunov irregular points  
is called a \emph{Lyapunov irregular set} of $f$.
Oseledets' theorem implies that any Lyapunov irregular set is negligible with respect to any $f$-invariant probability measures 
on $\mathrm{supp}(f)$.
So we consider here the Lebesgue measure instead of invariant measures.
From physical point of view, the Lebesgue measure  often plays an indispensable  role.

We say that a Lebesgue measurable subset $U$ of $M$ is \emph{observable} if 
$\mathrm{Leb}(U)>0$, where $\mathrm{Leb}(U)$ denotes the Lebesgue measure of $U$.
As far as the authors know, the first published example of dynamics with observable Lyapunov irregular set 
was presented by Ott and Yorke \cite{oy}.
They showed that the time-1 map of a figure-eight flow on $\mathbb{R}^2$ has a basin $U$ of attraction any interior points of which 
are Lyapunov irregular.
Their argument was 
complemented by \cite[Proposition 1.2]{klns} with the use of 
the concrete model of figure-eight flow presented by Guarino, Guih\'eneuf and Santiago \cite{ggs}.
Ott and Yorke also suggested that the time-1 map of a Bowen flow have an observable Lyapunov irregular set 
with supporting evidence of numerical experiment.
A strict proof of their suggestion was given by \cite[Proposition 1.1]{klns}.
Colli and Vargas \cite{cv} presented 2-dimensional diffeomorphisms $f:\mathbb{R}^2\longrightarrow \mathbb{R}^2$ 
which have a horseshoe set $\Lambda$ with a homoclinic tangency and a 
contracting wandering domain.
Here we say that a connected open set $U$ in $\mathbb{R}^2$ is a \emph{wandering domain} of $f$ if 
$f^m(U)\cap f^n(U)=\emptyset$ for any $m,n\in \mathbb{N}$ with $m\neq n$.
In \cite{klns},  Colli-Vargas' maps were modified so that they have Lyapunov irregular wandering domains.
On the other hand, it was shown in \cite{nnt}  that the Lyapunov irregular set has zero Lebesgue measure under certain type of noise.

We remark that, for any examples of dynamics $f$ and ${\boldsymbol{v}}\in T_xU\setminus \{\mathbf{0}\}$ given as above, $\sup \omega(\mathcal{A}_f(x,{\boldsymbol{v}}))$ is non-positive.
In this paper, we show that there exists a one-parameter family of piecewise expanding maps $F_\sigma:D\longrightarrow D$ on an open rectangle $D$ in $\mathbb{R}^2$ such that each $F_\sigma$ has an observable Lyapunov irregular set $U_\sigma$ 
with $\inf \omega(\mathcal{A}_{F_\sigma}(x,{\boldsymbol{v}}))>0$ for any $0<\sigma<1$, $x\in U_\sigma$ and ${\boldsymbol{v}}\in T_x(D)\setminus \{\mathbf{0}\}$.
Lyapunov irregular points for piecewise expanding maps are defined similarly 
as in the case of usual differentiable maps, see Subsection \ref{ss_PEM} for details.

\begin{mtheorem}\label{thm_A}
Suppose that $r$ is any integer with $1\leq r<\infty$.
Then there exists a one-parameter family of piecewise $\mathcal C^r$ expanding maps $F_\sigma:D\longrightarrow D$ $(0<\sigma<1)$ 
on an open rectangle $D$ 
which $\mathcal C^r$-converge to another expanding map as $\sigma\to 0$ and 
such that each $F_\sigma$ admits a wandering rectangle $R_\sigma\subset D$ such that, for Lebesgue-almost every $x\in R_\sigma$ and any non-zero ${\boldsymbol{v}}\in T_x D$, the $\omega$-limit set of $\mathcal{A}_{F_\sigma}(x,{\boldsymbol{v}})$ satisfies the following conditions.
\begin{enumerate}[$\bullet$]
\item
$\omega(\mathcal{A}_{F_\sigma}(x,{\boldsymbol{v}}))$ is a non-trivial closed interval 
which converges to a one-point set as $\sigma\to 0$.
\item
$\inf \omega(\mathcal{A}_{F_\sigma}(x,{\boldsymbol{v}}))\geq \rho$ for some constant $\rho>0$ independent of $\sigma$.
\end{enumerate}
\end{mtheorem}

Here an interval $I$ being \emph{non-trivial} means that $I$ is not a one-point set.
In particular, this theorem implies that the rectangle $R_\sigma$ contains an observable Lyapunov irregular subset.
Theorem \ref{thm_A} follows from Theorem \ref{thm_Tj_map} in Section \ref{Proof_thm_A}, where 
$\omega(\mathcal{A}_{F_\sigma}(x,{\boldsymbol{v}}))$ is presented exactly.

\begin{remark}[Lyapunov irregularity versus Birkhoff irregularity]
Let $f:M\longrightarrow M$ be a differentiable map 
on a manifold $M$ such that the support $\mathrm{supp}(f)$ of $f$ is compact.
A point $x$ of $\mathrm{supp}(f)$ is said to be \emph{Birkhoff irregular} with respect to $f$ if the empirical measure $\delta_x^{(n)}=n^{-1}\sum_{k=0}^{n-1} \delta_{f^k(x)}$ does not converge weakly as $n\to\infty$, 
otherwise $x$ is called \emph{Birkhoff regular}.
We say that a subset $U$ of $M$ consisting of Birkhoff (ir)regular points is 
a \emph{Birkhoff (ir)regular set} of $f$.
Birkhoff irregularity is an important criterion in the study of chaotic dynamics as well 
as Lyapunov irregularity.
In general, these criteria are independent of each other.
For example, the observable Lyapunov irregular set for the figure-eight flow given in \cite{klns} 
is Birkhoff regular.
On the other hand, Takens \cite{ta} proved that the observable Lyapunov irregular set for the Bowen flow as above 
is Birkhoff irregular.
As was shown in Theorem A in \cite{klns}, 
there exist two modified Colli-Vargas maps $f_1$, $f_2$ which are arbitrarily $\mathcal C^r$-close to each other 
and have a common observable Lyapunov irregular set  
that is Birkhoff regular with respect to $f_1$ and irregular with respect to $f_2$.
From the construction of the piecewise expanding maps $F_\sigma$ in Theorem \ref{thm_A}, 
it is not hard to see that the wandering rectangle $R_\sigma$ is Birkhoff regular.
More precisely, for any $x\in R_\sigma$, $\delta_x^{(n)}$ converges weakly to the Dirac measure at the origin $(0,0)$. 
\end{remark}

In strong contrast to Theorem \ref{thm_A},  we have the following theorem in the real analytic case.

\begin{mtheorem}\label{thm_B}
Let $F: D\longrightarrow D$ be a piecewise $\mathcal C^\omega$ expanding map on an open rectangle $D$.
Then any Lyapunov irregular set of $F$ has zero Lebesgue measure.
\end{mtheorem}

As we will see in Section \ref{S_wo_LIS}, the conclusion of Theorem \ref{thm_B} may hold in more general setting.
For example, one can literally repeat the argument in the proof of Theorem \ref{thm_B} 
if the  Perron--Frobenius operator of a piecewise expanding map satisfies the Lasota--Yorke inequality on a Banach space whose intersection with $L^\infty(\mathrm{Leb})=L^\infty(D,\mathrm{Leb})$ is dense in $L^\infty(\mathrm{Leb})$ (see Section \ref{S_wo_LIS} for precise definitions).
We will not pursue such generalizations because the main focus of the present paper is the Lyapunov irregular sets of piecewise expanding maps on the plane, 
but 
we just refer to   \cite{th,ba2} and references therein about geometric conditions under which  the  Perron--Frobenius operator of a piecewise $\mathcal C^2$ expanding map   of general dimension satisfies the Lasota--Yorke inequality. 
We finally remark that, since it is classically known after Lasota--Yorke \cite{ly} that the Perron--Frobenius operator of any 1-dimensional piecewise $\mathcal C^2$ expanding map satisfies the Lasota--Yorke inequality  on the space of bounded variation functions (whose intersection with $L^\infty(\mathrm{Leb})$ is dense in $L^\infty(\mathrm{Leb})$), 
any 1-dimensional piecewise $\mathcal C^2$ expanding map cannot have an observable Lyapunov irregular set    
unlike 2-dimensional maps in Theorem \ref{thm_A}.

\section{One-parameter families of piecewise expanding maps}\label{S_NCP}

\subsection{Piecewise expanding maps on planar domains}\label{ss_PEM}

Suppose that either $r$ is a positive integer or $r=\infty$ or $r=\omega$ 
and $D$ is a bounded region in $\mathbb{R}^2$ such that 
the topological boundary $\partial D$ of $D$ consists of finitely many simple regular $\mathcal C^r$-curves of finite 
length the interiors of which are mutually disjoint.
Let $\alpha_1,\dots,\alpha_u$ be simple regular $\mathcal C^r$-curves in the closure $\overline D=D\cup \partial D$ 
of $D$ with mutually disjoint interior and such that 
the union $\alpha_1\cup\cdots\cup\alpha_u$ is a 
graph in $\overline D$ any degree-one vertex of 
which is contained in $\partial D$.
Let $D_1,\dots,D_p$ be the components of $D'=D\setminus \alpha_1\cup\cdots\cup \alpha_u$.
A $\mathcal C^r$-map $F:D'\longrightarrow D$ is said to be a $\mathcal C^r$-\emph{piecewise expanding map} on $D$ if there exists a constant $\lambda>1$ and an open neighborhood $\mathcal{N}(\overline D_i)$ of $\overline D_i$ in $\mathbb{R}^2$ 
satisfying the following condition for any $i=1,\dots,p$.
\begin{enumerate}[$\bullet$]
\item
The restriction $F|_{D_i}$ is extended to a $\mathcal C^r$-diffeomorphism $F_i:\mathcal{N}(\overline D_i)\longrightarrow \mathbb{R}^2$ 
such that $\|DF_i(x){\boldsymbol{v}}\|\geq \lambda\|{\boldsymbol{v}}\|$ for any $x\in \mathcal{N}(\overline D_i)$ and ${\boldsymbol{v}}\in T_{x}\mathcal{N}(\overline D_i)$.
\end{enumerate}
Usually such a piecewise expanding map is denoted simply as $F:D\longrightarrow D$.
It follows from the definition that 
\begin{equation}\label{eqn_LebU}
\mathrm{Leb}(U)=\mathrm{Leb}\biggl(U\cap \bigcap_{n=0}^\infty F^{-n}(D')\biggr)
\end{equation}
holds for any open set $U$ of $D$.
A point $x$ of $D$ is called \emph{Lyapunov irregular} if 
$x\in \bigcap_{n=0}^\infty F^{-n}(D')$ and $\omega(\mathcal{A}_F(x,{\boldsymbol{v}}))$ is not a single point set 
for some non-zero ${\boldsymbol{v}}\in T_{x}(D)$.

\subsection{Non-conformal piecewise expanding maps}
Piecewise expanding maps used in the proof of Theorem \ref{thm_A} are modifications of that in 
Tsujii \cite{tj}.
Tsujii's map is conformal on the orbit of the concerned wandering domain, while ours are non-conformal.

Throughout the remainder of this section, suppose that $1\leq r<\infty$ 
and $D$ is the open rectangle $(0,1)\times (-1,1)$. 
We will define a 
one-parameter family of piecewise $\mathcal C^r$-expanding maps $F_\sigma:D\longrightarrow D$ $(0< \sigma<1)$ $\mathcal C^r$-converging to the 
piecewise expanding map in \cite{tj}.

Fix an integer $\gamma>r$ and a small $\varepsilon_1>0$ with 
\begin{equation}\label{eqn_ve_gamma}
\varepsilon_1<\gamma^{-3}(2-r^{-1})^2\log 2.
\end{equation}
We set $\lambda_1=\exp(\varepsilon_1)$ and take $\lambda_2=\exp(\varepsilon_2)$ with $1<\lambda_2<\lambda_1$ arbitrarily.
Our piecewise expanding maps depend on $\varepsilon_2$ for a fixed $\varepsilon_1$.
Suppose that $\sigma=1-\varepsilon_2/\varepsilon_1$ is the parameter of $F_{\varepsilon_2}$.
Then $\sigma\searrow 0$ is equivalent to $\varepsilon_2\nearrow \varepsilon_1$.
The map $F_{\varepsilon_2}$ is rewritten as $F_\sigma$ or $F$ shortly.

Consider a $\mathcal C^r$-function $f=f_\sigma:[0,\exp(-\varepsilon_1)]\longrightarrow [0,\exp(-\varepsilon_2))$ 
with $f(0)=f(\exp(-\varepsilon_1))=0$ and the splitting of $D$ into the following eight regions, 
where $\delta=\pm$.
\begin{align*}
D_{0,\delta}&=\{(x,y)\in \mathbb{R}^2\,|\, 0<x<\exp(-\varepsilon_1), \delta f(x)<\delta y<\exp(-\varepsilon_2)\},\\
D_{1,\delta}&=\{(x,y)\in \mathbb{R}^2\,|\, \exp(-\varepsilon_1)<x<1, 0<\delta y<\exp(-\varepsilon_2)\},\\
D_{2,\delta}&=\{(x,y)\in \mathbb{R}^2\,|\, 0<x<\exp(-\varepsilon_1), \exp(-\varepsilon_2)<\delta y<1\},\\
D_{3,\delta}&=\{(x,y)\in \mathbb{R}^2\,|\, \exp(-\varepsilon_1)<x<1, \exp(-\varepsilon_2)<\delta y<1\}.
\end{align*}
See Figure \ref{f_part_D}.
\begin{figure}[hbtp]
\centering
\includegraphics[width=71.3mm]{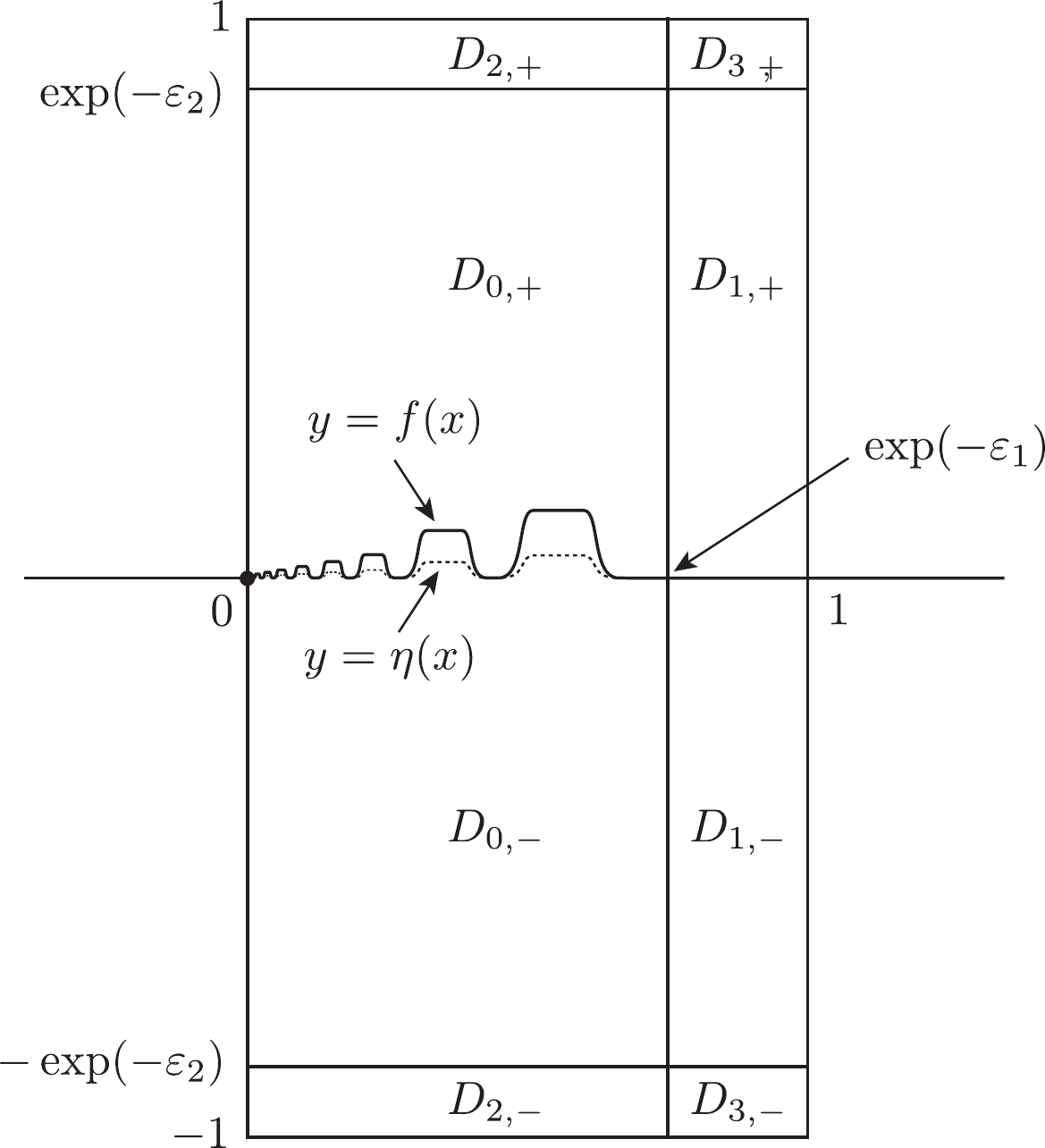}
\caption{Splitting of $D$.}
\label{f_part_D}
\end{figure}

We need to consider another $\mathcal C^r$-function $\eta=\eta_\sigma:[0,\exp(-\varepsilon_1)]\longrightarrow [0,1]$ with $\eta(x)\leq f(x)$ for 
any $x\in [0,\exp(\varepsilon_1)]$.
Then a piecewise $\mathcal C^r$-expanding map $F=F_\sigma:D\longrightarrow D$ is defined as follows.
\begin{equation}\label{eqn_F}
F(x,y)=
\begin{cases}
(\lambda_1 x,\lambda_2(y-\eta(x)))&\text{for }(x,y)\in D_{0,+},\\
(\lambda_1 x,\lambda_2 y)&\text{for }(x,y)\in D_{0,-},\\
(\delta\lambda_2 y,\lambda_1(x-\exp(-\varepsilon_1)))&\text{for }(x,y)\in D_{1,\delta},\\
(\lambda_1 x,\lambda_2(y-\delta \exp(-\varepsilon_2)))&\text{for }(x,y)\in D_{2,\delta},\\
(\lambda_1(x-\exp(-\varepsilon_1)),\lambda_2(y-\delta\exp(-\varepsilon_2)))&\text{for }(x,y)\in D_{3,\delta}.
\end{cases}
\end{equation}
Let ${\boldsymbol{e}}_1=(1,0)$, ${\boldsymbol{e}}_2=(0,1)$ be the unit vectors tangent to $D$ at ${\boldsymbol{x}}=(x,y)\in D$ under the 
natural identification of $T_{\boldsymbol{x}} D=\mathbb{R}^2$.
By \eqref{eqn_F}, 
for any ${\boldsymbol{v}}=\alpha{\boldsymbol{e}}_1+\beta{\boldsymbol{e}}_2\in T_{{\boldsymbol{x}}}D$ with ${\boldsymbol{x}}\in D_{0,+}\cup D_{0,-}\cup D_{1,+}$, 
we have 
\begin{equation}\label{eqn_DF}
DF({\boldsymbol{x}}){\boldsymbol{v}}=
\begin{cases}
\alpha\lambda_1 {\boldsymbol{e}}_1+\left(-\alpha\dfrac{d\eta}{dx}(x)+\beta\right)\lambda_2{\boldsymbol{e}}_2&\text{for }{\boldsymbol{x}}\in D_{0,+},\\
\alpha\lambda_1 {\boldsymbol{e}}_1+\beta\lambda_2{\boldsymbol{e}}_2&\text{for }{\boldsymbol{x}}\in D_{0,-},\\
\beta\lambda_2{\boldsymbol{e}}_1+\alpha\lambda_1 {\boldsymbol{e}}_2&\text{for }{\boldsymbol{x}}\in D_{1,+}.
\end{cases}
\end{equation}

\subsection{Bottom, left-side, height and width sequences}
Now we modify the four double-indexed sequences given 
in \cite{tj}, which are used to define a wandering rectangle of $F$.

Suppose that $M$ is a sufficiently large integer, which will be fixed later.
For any integers $m$, $j$ with $m>M$, let 
$$\zeta_{2j-1}^{(m)}=\varepsilon_1\gamma^m\quad\text{and}\quad\zeta_{2j}^{(m)}=\varepsilon_2\gamma^m.$$
Then we set  
\begin{align*}
B(k,m)&=\exp(\varepsilon_2k)\biggl(\,\sum_{i=0}^\infty\exp\biggl(-\sum_{j=0}^{2i+1}\zeta_j^{(m+j)}\biggr)\biggr),\\
L(k,m)&=\exp(\varepsilon_1k)\biggl(\exp\bigl(-\zeta_1^{(m)}\bigr)+\sum_{i=0}^\infty \exp\biggl(-\sum_{j=0}^{2i+2}\zeta_{j+1}^{(m+j)}\biggr)\biggr)
\end{align*}
for any $0\leq k\leq \gamma^m-1$.
It follows from the definitions that  
\begin{align*}
B(k+1,m)&=\lambda_2B(k,m),\quad L(k+1,m)=\lambda_1 L(k,m)\quad \text{for }\ 0\leq k\leq\gamma^m-2,\\
B(0,m+1)&=\lambda_1(L(\gamma^m-1,m)-\exp(-\varepsilon_1)),\\
L(0,m+1)&=\lambda_2 B(\gamma^m-1,m).
\end{align*}
One can take $M$ so that, for any $m>M$ and $1\leq k\leq \gamma^m$, the inequalities 
\begin{align*}
&B(k,m)\leq B(\gamma^m-1,m)<\exp (-\varepsilon_1\gamma^{m+1}),\quad B(\gamma^{m+1}-k,m+1)<\frac{B(\gamma^m-k,m)}2,\\
&0<L(\gamma^m-k,m)-\exp(-\varepsilon_1k)
<\Delta(k-1)
\end{align*}
hold, where 
$$\Delta(k-1)=\frac{\exp(-\varepsilon_1(k-1))-\exp(-\varepsilon_1k)}6.$$
We also consider the following sequences.
\begin{align*}
H(k,m)&=\exp\biggl((\varepsilon_2-\log 2)k+\sum_{i=0}^{m-1}\zeta_{m-i}^{(i)}-(1-\gamma^{-1})\log 2\sum_{i=1}^{[m/2]}\gamma^{m-2i}\biggr)\\
&\hspace{180pt}\text{for}\quad 0\leq k<\gamma^m-\gamma^{m-1},\\
H(k,m)&=\exp\biggl(\varepsilon_2k+\sum_{i=0}^{m-1}\zeta_{m-i}^{(i)}-(1-\gamma^{-1})\log 2\sum_{i=1}^{[m/2]}\gamma^{m-2i}\biggr)\\
&\hspace{180pt}\text{for}\quad \gamma^m-\gamma^{m-1}\leq k\leq \gamma^m-1,\\
W(k,m)&=\exp\biggl(\varepsilon_1k+\sum_{i=0}^{m-1}\zeta_{m-1-i}^{(i)}-(1-\gamma^{-1})\log 2\sum_{i=1}^{[(m-1)/2]}\gamma^{m-2i-1}\biggr)\\
&\hspace{180pt}\text{for}\quad 0\leq k\leq \gamma^m-1.
\end{align*}
It follows from the definitions that
\begin{align*}
W(k+1,m)&=\lambda_1 W(k,m)\qquad\text{for}\quad 0\leq k\leq \gamma^m-2,\\
H(k+1,m)&=\dfrac{\lambda_2}2H(k,m)\qquad\text{for}\quad 0\leq k< \gamma^m-\gamma^{m-1},\\
H(k+1,m)&=\lambda_2H(k,m)\qquad\hspace{2.5pt}\text{for}\quad \gamma^m-\gamma^{m-1}\leq k\leq \gamma^m-2,\\
W(0,m+1)&=\lambda_2 H(\gamma^m-1,m),\quad H(0,m+1)=\lambda_1 W(\gamma^m-1,m).
\end{align*}

By the condition \eqref{eqn_ve_gamma} on $\varepsilon_1$, 
if necessary retaking $M$ by a larger integer, we may assume that 
\begin{align*}
W(\gamma^m-k,m)&\leq \exp(-\varepsilon_1\gamma^{m+2})<\Delta(k-1)\quad\text{and}\\
H(\gamma^m-k,m)&<B(\gamma^m-k,m)<\exp(-\varepsilon_1\gamma^{m+1})
\end{align*}
hold for any $1\leq k\leq \gamma^m$.

Consider the closed intervals 
$I_k=[\exp(-\varepsilon_1k),\exp(-\varepsilon_1k)+6\Delta(k-1)]$ and $J_k=[\exp(-\varepsilon_1k),\exp(-\varepsilon_1k)+2\Delta(k-1)]$.
Let $K_k$ be the $\Delta(k)$-neighborhood of $J_k$ in $\mathbb{R}$.
Then $K_k$ $(k=1,2,\dots)$ are mutually disjoint closed intervals.
If we take $M$ sufficiently large, then we can have functions $f$, $\eta$ as above 
satisfying the following conditions.
\begin{enumerate}[$\bullet$]
\item
$f(x)=0$ and $\eta(x)=0$ if $\exp(-\varepsilon_1\gamma^M)\leq x\leq \exp(-\varepsilon_1)$.
\item
$f(x)=0$ and $\eta(x)=0$ if $x\not\in \bigcup_{k=1}^\infty K_k$.
\item
For any $x\in J_k$ and $k>\gamma^M$,
\begin{equation}\label{eqn_etaH}
\begin{split}
f(x)&=B(\gamma^{m(k)}-k,m(k))+\frac12 H(\gamma^{m(k)}-k,m(k)),\\
\eta(x)&=\frac12 H(\gamma^{m(k)}-k,m(k)),
\end{split}
\end{equation}
where $m(x)$ is the integer with $\gamma^{m(k)-1}<k\leq \gamma^{m(k)}$.
\item
Both $f$ and $\eta$ $\mathcal C^r$-depend on the parameter $\sigma=1-\varepsilon_2/\varepsilon_1$.
\end{enumerate}

\section{Maps with observable Lyapunov irregular sets}\label{Proof_thm_A}

Let $F=F_\sigma:D\longrightarrow D$ be the $\mathcal C^r$-piecewise expanding maps with $1\leq r<\infty$ and 
$m\geq 1$ the integer given in Section \ref{S_NCP}.
In this section, we will show that 
$F$ has a Lyapunov irregular wandering rectangle $R=R_\sigma$ satisfying the 
conditions required in Theorem \ref{thm_A}.

\subsection{Wandering rectangles}\label{S_O_LIS}

For any $0\leq k\leq \gamma^m$, we set 
\begin{align*}
R(k,m)&=\left\{\boldsymbol{x}\in \mathbb{R}^2\,\biggl|\ 
\begin{matrix}
B(k,m)<y<B(k,m)+H(k,m),\\ L(k,m)<x<L(k,m)+W(k,m)
\end{matrix}\ 
\right\},\\
R(k,m)'&=\bigl\{\,\boldsymbol{x}\in R(k,m)\,|\, F^n({\boldsymbol{x}})\in D'\ (n=0,1,\dots)\,\bigr\},
\end{align*}
where $\boldsymbol{x}=(x,y)$ and $D'$ is the disjoint union of the eight regions $D_{0,+},\cdots,D_{3,-}$.
By using an argument similar to that in \cite[Section 2]{tj}, 
one can take $M$ so large that $R(k,m)$ is a wandering domain of $F$ if $m>M$.
By \eqref{eqn_LebU}, $\mathrm{Leb}(R(0,m)')=\mathrm{Leb}(R(0,m))>0$.
By \eqref{eqn_etaH}, $\eta(x)$ is constant on the interval $[L(k,p),L(k,p)+W(k,p)]$ for $p=m,m+1,\dots$ and $k=0,1,\dots,\gamma^p-1$.
Since 
$R(k,p)'\subset [L(k,p),L(k,p)+W(k,p)]\times (-1,1)$,  
it follows from \eqref{eqn_DF} that    
\begin{equation}\label{eqn_DF2}
DF({\boldsymbol{x}}){\boldsymbol{v}}=\alpha\lambda_1 {\boldsymbol{e}}_1+\beta\lambda_2{\boldsymbol{e}}_2
\end{equation}
for any ${\boldsymbol{x}}\in R(k,p)'$ with $p,k$ as above and any ${\boldsymbol{v}}=\alpha{\boldsymbol{e}}_1+\beta{\boldsymbol{e}}_2\in T_{{\boldsymbol{x}}}D$.
Set 
$$\xi_0=\frac{\gamma}{\gamma+1}\log\lambda_1+\frac1{\gamma+1}\log\lambda_2,\quad  
\xi_1=\frac{1}{\gamma+1}\log\lambda_1+\frac{\gamma}{\gamma+1}\log\lambda_2.$$
Since $\lambda_1>\lambda_2>1$ and $\gamma>1$, 
it follows that 
$$\xi_1<\log\sqrt{\lambda_1\lambda_2}<\xi_0.$$

Recall the definition \eqref{eqn_Af} such that $\mathcal{A}_F({\boldsymbol{x}},{\boldsymbol{v}})$ is the sequence defined as  
$$\mathcal{A}_F({\boldsymbol{x}},{\boldsymbol{v}})=\left\{ \dfrac1{n}\log\|DF^n({\boldsymbol{x}}){\boldsymbol{v}}\|\ \Big|\ n=1,2,\dots\right\}$$
for any ${\boldsymbol{x}}\in R(0,m)'$ and ${\boldsymbol{v}}\in T_{{\boldsymbol{x}}}D\setminus \{\mathbf{0}\}$.

\subsection{Proof of Theorem \ref{thm_A}}

The following theorem implies Theorem \ref{thm_A}, where
$R(0,m)$ corresponds to the wandering rectangle $R_\sigma$ in Theorem \ref{thm_A}.

\begin{theorem}\label{thm_Tj_map}
For any  ${\boldsymbol{x}}\in R(0,m)'$, suppose that ${\boldsymbol{v}}=\alpha{\boldsymbol{e}}_1+\beta{\boldsymbol{e}}_2$ is any non-zero element of $T_{{\boldsymbol{x}}}D$.
Then 
\begin{enumerate}[\rm ({A}1)]
\item
$\omega(\mathcal{A}_F({\boldsymbol{x}},{\boldsymbol{v}}))=\left[\log\sqrt{\lambda_1\lambda_2},\xi_0\right]$ if $\alpha\beta\neq 0$,\\[-6pt]
\item
$\omega(\mathcal{A}_F({\boldsymbol{x}},{\boldsymbol{v}}))=\left[\xi_1,\xi_0\right]$ if $\alpha\beta=0$.
\end{enumerate}
\end{theorem}

First of all, we consider the subsequences $\mathcal{A}_F^{(1)}({\boldsymbol{x}},{\boldsymbol{v}})$, $\mathcal{A}_F^{(2)}({\boldsymbol{x}},{\boldsymbol{v}})$ 
of $\mathcal{A}_F({\boldsymbol{v}})$ defined as
\begin{align*}
\mathcal{A}_F^{(1)}({\boldsymbol{x}},{\boldsymbol{v}})&=\biggl\{ \dfrac1{a}\log\|DF^a({\boldsymbol{x}}){\boldsymbol{v}}\|\ \Big|\ a\in \biggl(\,\bigcup_{n=1}^\infty
[\gamma^{m+2n-1},\gamma^{m+2n}]\,\biggr)\cap \mathbb{N}\biggr\},\\
\mathcal{A}_F^{(2)}({\boldsymbol{x}},{\boldsymbol{v}})&=\biggl\{ \dfrac1{b}\log\|DF^b({\boldsymbol{x}}){\boldsymbol{v}}\|\ \Big|\ b\in \biggl(\,\bigcup_{n=1}^\infty
[\gamma^{m+2n},\gamma^{m+2n+1}]\,\biggr)\cap \mathbb{N}\biggr\}.
\end{align*}
The assertion (A1) of Theorem \ref{thm_Tj_map} is a corollary to the following two lemmas.

\begin{lemma}\label{l_ab_neq0}
$\omega(\mathcal{A}_F^{(1)}({\boldsymbol{x}},{\boldsymbol{v}}))=\bigl[\log\sqrt{\lambda_1\lambda_2},\xi_0\bigr]$ if $\alpha\beta\neq 0$.
\end{lemma}
\begin{proof}
For any positive integer $n$, 
we set 
$$
u_{m,n}=\gamma^m(-1+\gamma-\gamma^2+\cdots-\gamma^{2n-2}+\gamma^{2n-1})=
\gamma^m(\gamma-1)\sum_{j=0}^n \gamma^{2j}.
$$
By \eqref{eqn_DF} and \eqref{eqn_DF2}, 
\begin{align*}
DF^{\gamma^m\sum_{i=0}^{2n-1}\gamma^i+k}({\boldsymbol{x}}){\boldsymbol{v}}
&=\alpha\lambda_1^{\gamma^m\sum_{j=0}^{n-1}\gamma^{2j}+k}\lambda_2^{\gamma^{m+1}\sum_{j=0}^{n-1}\gamma^{2j}}
{\boldsymbol{e}}_1\\
&\hspace{60pt}
+\beta\lambda_1^{\gamma^{m+1}\sum_{j=0}^{n-1}\gamma^{2j}}\lambda_2^{\gamma^m\sum_{j=0}^{n-1}\gamma^{2j}+k}
{\boldsymbol{e}}_2
\end{align*}
for any $k=0,1,\dots,\gamma^{m+2n}-1$.
The equation in the case of $k=\gamma^{m+2n}$ is obtained from the above equation by swapping ${\boldsymbol{e}}_1$ for ${\boldsymbol{e}}_2$.
Thus, for any $k=0,1,\dots,\gamma^{m+2n}$, 
\begin{equation}\label{eqn_log_DF0}
\begin{split}
\log \|DF^{\gamma^m\sum_{i=0}^{2n-1}\gamma^i+k}({\boldsymbol{x}}){\boldsymbol{v}}\|
=\frac12\log\Bigl(&\alpha^2\lambda_1^{2\gamma^m\sum_{j=0}^{n-1}\gamma^{2j}+2k}\lambda_2^{2\gamma^{m+1}\sum_{j=0}^{n-1}
\gamma^{2j}}\\
&+\beta^2\lambda_1^{2\gamma^{m+1}\sum_{j=0}^{n-1}\gamma^{2j}}\lambda_2^{2\gamma^m\sum_{j=0}^{n-1}\gamma^{2j}+2k}\Bigr).
\end{split}
\end{equation}
Let $c_{n,k}$ be the number defined as 
\begin{equation}\label{eqn_c_nk}
c_{n,k}=\begin{cases}
\alpha^2(\lambda_1^{-1}\lambda_2)^{2(u_{m,n}-k)}+\beta^2& \text{for $k=0,1,\dots,u_{n,m}$},\\
\alpha^2+\beta^2(\lambda_1^{-1}\lambda_2)^{2(k-u_{m,n})}& \text{for $k=u_{n,m},u_{n,m}+1,\dots,\gamma^{m+2n}$}.
\end{cases}
\end{equation}
By \eqref{eqn_log_DF0}, for any $k=0,1,\dots,u_{n,m}$, 
\begin{equation}\label{eqn_log_DF1}
\begin{split}
\log \|D&F^{\gamma^m\sum_{i=0}^{2n-1}\gamma^i+k}({\boldsymbol{x}}){\boldsymbol{v}}\|=
\frac12\log\Big(\lambda_1^{2\gamma^{m+1}\sum_{j=0}^{n-1}\gamma^{2j}}\lambda_2^{2\gamma^m\sum_{j=0}^{n-1}\gamma^{2j}+2k}
c_{n,k}\Bigr)\\
&=\Bigl(\gamma^{m+1}\sum_{j=0}^{n-1}\gamma^{2j}\Bigr)\log\lambda_1+
\Bigl(\gamma^{m}\sum\limits_{j=0}^{n-1}\gamma^{2j}+k\Bigr)\log \lambda_2
+\frac12\log c_{n,k}.
\end{split}
\end{equation}
Similarly, for any $k=u_{n,m},u_{n,m}+1,\dots,\gamma^{m+2n}$, 
\begin{equation}\label{eqn_log_DF2}
\begin{split}
\log \|D&F^{\gamma^m\sum_{i=0}^{2n-1}\gamma^i+k}({\boldsymbol{x}}){\boldsymbol{v}}\|=
\frac12\log\Big(\lambda_1^{2\gamma^{m+1}\sum_{j=0}^{n-1}\gamma^{2j}+2k}\lambda_2^{2\gamma^m\sum_{j=0}^{n-1}\gamma^{2j}}
c_{n,k}\Bigr)\\
&=\Bigl(\gamma^{m+1}\sum_{j=0}^{n-1}\gamma^{2j}+k\Bigr)\log\lambda_1+
\Bigl(\gamma^{m}\sum\limits_{j=0}^{n-1}\gamma^{2j}\Bigr)\log \lambda_2
+\frac12\log c_{n,k}.
\end{split}
\end{equation}
Since $0<\lambda_1^{-1}\lambda_2<1$, we have by \eqref{eqn_c_nk} 
$$\min\{\log \alpha^2,\log \beta^2\}\leq \log c_{n,k}\leq \log (\alpha^2+\beta^2).$$
Hence, for any $k=0,1,\dots,\gamma^{m+2n}$, 
\begin{equation}\label{eqn_alpha_beta2}
|\log c_{n,k}|\leq \max\bigl\{\log (\alpha^2+\beta^2),\, \log \alpha^{-2},\, \log \beta^{-2}\bigr\}.
\end{equation}
We define the function $\varphi_n:[0,\gamma^{m+2n}]\longrightarrow \mathbb{R}$ by
$$\varphi_n(t)=\begin{cases}
\varphi_n^{(1)}(t)&\text{for $0\leq t\leq u_{m,n}$},\\[4pt]
\varphi_n^{(2)}(t)&\text{for $u_{m,n}< t\leq \gamma^{m+2n}$},
\end{cases}$$ 
where 
\begin{subequations}
\begin{align}
\varphi_n^{(1)}(t)&=\dfrac{
\Bigl(\gamma^{m+1}\sum_{j=0}^{n-1}\gamma^{2j}\Bigr)\log\lambda_1+
\Bigl(\gamma^{m}\sum_{j=0}^{n-1}\gamma^{2j}+t\Bigr)\log \lambda_2}
{\gamma^m\sum_{i=0}^{2n-1}\gamma^i+t},\label{eqn_DF_3a}\\
\varphi_n^{(2)}(t)&=\dfrac{
\Bigl(\gamma^{m}\sum_{j=0}^{n-1}\gamma^{2j}+t\Bigr)\log\lambda_1+
\Bigl(\gamma^{m+1}\sum_{j=0}^{n-1}\gamma^{2j}\Bigr)\log \lambda_2}
{\gamma^m\sum_{i=0}^{2n-1}\gamma^i+t}\label{eqn_DF_3b}.
\end{align}
\end{subequations}
Since $\lim\limits_{t\to u_{m,n}+0}\varphi_{n}^{(2)}(t)=\log\sqrt{\lambda_1\lambda_2}=\varphi_n^{(1)}(u_{m,n})$, 
$\varphi_n(t)$ is a continuous function.
From the definition, we know that $\varphi_n(t)$ is monotone decreasing on $0\leq t\leq u_{m,n}$ 
and monotone increasing on $u_{m,n}\leq  t\leq \gamma^{m+2n}$ and 
$\varphi_{n}(0)=\varphi_n(\gamma^{m+2n}-\gamma^m)=\xi_0$ holds.
Note that $\varphi_n(t)$ is differentiable at any $t\neq u_{m,n}$ and satisfies 
$$
\sup\biggl\{\,\Bigl|\frac{d\varphi_n}{dt}(t)\Bigr|\ \bigg|\ t\in [0,\gamma^{m+2n}]\setminus \{u_{m,n}\}\,\biggr\}=O(\gamma^{-2n}).
$$
By the mean value theorem, for any $\xi\in \bigl[\log\sqrt{\lambda_1\lambda_2},\xi_0\bigr]$, 
\begin{equation}\label{eqn_vp-xi}
\varphi_n(\gamma^{m+2n})-\xi_0=O(\gamma^{-2n})(\gamma^{m+2n}-(\gamma^{m+2n}-\gamma^m))=O(\gamma^{-2n}).
\end{equation}
By the intermediate value theorem, 
for any $\xi\in \bigl[\log\sqrt{\lambda_1\lambda_2},\xi_0\bigr]$, there exits a real number 
$t_{n}$ $(0\leq t_{n}\leq \gamma^{m+2n}-\gamma^m)$ with $\varphi_{n}(t_{n})=\xi$.
Again by using the mean value theorem, we have
$$
|\,\varphi_{n}(\lfloor t_{n} \rfloor)-\xi\,|=O(\gamma^{-2n}),
$$
where $\lfloor t_{n} \rfloor$ is the greatest integer not exceeding $t_n$.
This implies that $\lim\limits_{n\to\infty}\varphi_{n}(\lfloor t_{n} \rfloor)=\xi$.
Here we set $q_{n}=\gamma^m\sum_{i=0}^{2n-1}\gamma^i+\lfloor t_{n} \rfloor$.
By \eqref{eqn_log_DF1}, \eqref{eqn_log_DF2} and \eqref{eqn_alpha_beta2}, 
$$
\lim_{n\to\infty}\frac{\log\|DF^{q_{n}}({\boldsymbol{x}}){\boldsymbol{v}}\|}{q_{n}}=\lim_{n\to\infty}\varphi_{n}(\lfloor 
t_{n} \rfloor)
+\lim_{n\to\infty}\frac{\log(c_{{n},\lfloor t_{n} \rfloor})}{2q_{n}}=\xi. 
$$
It follows that 
$$\omega(\mathcal{A}_F^{(1)}({\boldsymbol{x}},{\boldsymbol{v}}))\supset \bigl[\log\sqrt{\lambda_1\lambda_2},\xi_0\bigr].$$

On the other hand, by \eqref{eqn_vp-xi}, $\varphi_n([0,\gamma^{m+2n}])=\bigl[\log\sqrt{\lambda_1\lambda_2},\xi_0+O(\gamma^{-2n})\bigr]$.
This fact shows $\omega(\mathcal{A}_F^{(1)}({\boldsymbol{x}},{\boldsymbol{v}}))\subset \bigl[\log\sqrt{\lambda_1\lambda_2},\xi_0\bigr]$ 
and hence $\omega(\mathcal{A}_F^{(1)}({\boldsymbol{x}},{\boldsymbol{v}}))= \bigl[\log\sqrt{\lambda_1\lambda_2},\xi_0\bigr]$, 
which completes the proof.
\end{proof}

\begin{lemma}\label{l_ab_neq1}
$\omega(\mathcal{A}_F^{(2)}({\boldsymbol{x}},{\boldsymbol{v}}))=\bigl[\log\sqrt{\lambda_1\lambda_2},\xi_0\bigr]$ if $\alpha\beta\neq 0$.
\end{lemma}
\begin{proof}
Let $b$ be any element of 
$[\gamma^{m+2n},\gamma^{m+2n+1}]\cap \mathbb{N}=[\gamma^{(m+1)+2n-1},\gamma^{(m+1)+2n}]\cap \mathbb{N}$.
If we set ${\boldsymbol{x}}'=F^{\gamma^m}({\boldsymbol{x}})$ and ${\boldsymbol{v}}'=\alpha'{\boldsymbol{e}}_1+\beta'{\boldsymbol{e}}_2=DF^{\gamma^m}({\boldsymbol{x}}){\boldsymbol{v}}$, 
then ${\boldsymbol{x}}'\in R(m+1,0)'$, ${\boldsymbol{v}}'\in T_{{\boldsymbol{x}}'}D$ and $\alpha'\beta'\neq 0$.
By the chain rule, 
$$DF^b({\boldsymbol{x}}'){\boldsymbol{v}}'=DF^{b+\gamma^m}({\boldsymbol{x}}){\boldsymbol{v}}=DF^{\gamma^m}(F^{b}({\boldsymbol{x}}))DF^b({\boldsymbol{x}}){\boldsymbol{v}}.$$
It follows from this fact together with \eqref{eqn_DF} and \eqref{eqn_DF2} that 
$$\lambda_2^{\gamma^m}\|DF^b({\boldsymbol{x}}){\boldsymbol{v}}\|
\leq \|DF^b({\boldsymbol{x}}'){\boldsymbol{v}}'\|
\leq \lambda_1^{\gamma^m}\|DF^b({\boldsymbol{x}}){\boldsymbol{v}}\|.$$
This shows that 
$$\frac{\log \|DF^b({\boldsymbol{x}}){\boldsymbol{v}}\|}{b}=\frac{\log \|DF^b({\boldsymbol{x}}'){\boldsymbol{v}}'\|}{b}+O(b^{-1}).$$
Hence, by Lemma \ref{l_ab_neq0},  
$\omega(\mathcal{A}_F^{(2)}({\boldsymbol{x}},{\boldsymbol{v}}))=\omega(\mathcal{A}_F^{(1)}({\boldsymbol{x}}',{\boldsymbol{v}}'))=\bigl[\log\sqrt{\lambda_1\lambda_2},\xi_0\bigr]$.
\end{proof}

By Lemmas \ref{l_ab_neq0} and \ref{l_ab_neq1}, we have 
$\omega(\mathcal{A}_F({\boldsymbol{x}},{\boldsymbol{v}}))=\left[\log \sqrt{\lambda_1\lambda_2},\xi_0\right]$ if 
$\alpha\beta\neq 0$, 
which shows the assertion (A1) of Theorem \ref{thm_Tj_map}.

\begin{proof}[Proof of the assertion (A2) of Theorem \ref{thm_Tj_map}.]
Now we consider the case of $\alpha\beta=0$.
Since ${\boldsymbol{v}}\neq \mathbf{0}$, one of $\alpha$ and $\beta$ is non-zero.
Thus it suffices to consider the case of either $(\alpha,\beta)=(0,1)$ or $(1,0)$.
By \eqref{eqn_log_DF0}, for any $k=0,1,\dots,\gamma^{m+2n}$, 
$$
\log \|DF^{\gamma^m\sum_{i=0}^{2n-1}\gamma^i+k}({\boldsymbol{x}}){\boldsymbol{v}}\|=
\begin{cases}
\dfrac12\log\Bigl(\lambda_1^{2\gamma^{m+1}\sum_{j=0}^{n-1}\gamma^{2j}}\lambda_2^{2\gamma^m\sum_{j=0}^{n-1}\gamma^{2j}+2k}\Bigr)\\
\hspace{120pt}\text{if $(\alpha,\beta)=(0,1)$},\\[10pt]
\dfrac12\log\Bigl(\lambda_1^{2\gamma^m\sum_{j=0}^{n-1}\gamma^{2j}+2k}\lambda_2^{2\gamma^{m+1}\sum_{j=0}^{n-1}\gamma^{2j}}\Bigr)\\
\hspace{120pt}\text{if $(\alpha,\beta)=(1,0)$}.
\end{cases}
$$
The domains of the functions $\varphi_n^{(1)}(t)$, $\varphi_n^{(2)}(t)$ given in \eqref{eqn_DF_3a} and 
\eqref{eqn_DF_3b} can be extended to $0\leq t\leq \gamma^{m+2n}$.
On the new domain, $\varphi_n^{(1)}(t)$ is monotone decreasing and $\varphi_n^{(2)}(t)$ 
is monotone increasing.
Furthermore  
$$\varphi_n^{(1)}(0)=\varphi_n^{(2)}(\gamma^{m+2n}-\gamma^m)=\xi_0\quad \text{and}\quad
\varphi_n^{(2)}(0)=\varphi_n^{(1)}(\gamma^{m+2n}-\gamma^m)=\xi_1$$
hold.
Thus, in either case of $(\alpha,\beta)=(0,1)$ or $(1,0)$, 
we have $\omega(\mathcal{A}_F^{(1)}({\boldsymbol{x}},{\boldsymbol{v}}))=[\xi_1,\xi_0]$ by using an argument similar to that in the proof of Lemma \ref{l_ab_neq0}.

In the case 
when $b\in [\gamma^{m+2n},\gamma^{m+2n+1}]\cap \mathbb{N}=[\gamma^{(m+1)+2n-1},\gamma^{(m+1)+2n}]\cap \mathbb{N}$, 
we set ${\boldsymbol{x}}'=F^{\gamma^m}({\boldsymbol{x}})$ and ${\boldsymbol{v}}'=DF^{\gamma^m}({\boldsymbol{x}}){\boldsymbol{v}}$.
Then, as in the proof of Lemma \ref{l_ab_neq1}, one can prove that 
$\omega(\mathcal{A}_F^{(2)}({\boldsymbol{x}},{\boldsymbol{v}}))=\omega(\mathcal{A}_F^{(1)}({\boldsymbol{x}}',{\boldsymbol{v}}'))=[\xi_1,\xi_0]$.
So we have $\omega(\mathcal{A}_F({\boldsymbol{x}},{\boldsymbol{v}}))=[\xi_1,\xi_0]$.
This completes the proof.
\end{proof}

\section{Maps without observable Lyapunov irregular sets}\label{S_wo_LIS}

In this section we give the proof of Theorem \ref{thm_B}. 
Throughout this section, we suppose that $F:D\longrightarrow D$ is a piecewise $\mathcal C^\omega$ expanding map on an open rectangle $D$ 
and $D'=D_1\cup\cdots\cup D_p$ is a splitting of $D$ (modulo zero Lebesgue  measure sets) satisfying the conditions in Subsection \ref{ss_PEM}.
In particular, $F|_{D_i}$ is extended to a $\mathcal C^\omega$-diffeomorphism $F_i:\mathcal{N}(\overline D_i)\longrightarrow \mathbb{R}^2$ 
for $i=1,\dots,p$.
We denote the support of a Borel probability measure $\nu$ by $\mathrm{supp} (\nu )$.
We first show the following key proposition, which enables us to apply the classical Oseledets multiplicative ergodic theorem. 

\begin{prop}\label{prop:0317}
There are finitely many $\mathrm{Leb}$-absolutely continuous ergodic invariant probability measures $\mu_1,\ldots , \mu_q$ such that for $\mathrm{Leb}$-almost every ${x}\in D$, 
 there exists an integer $N$ such that $F^N({x}) \in \bigcup_{j=1}^q \mathrm{supp}(\mu_j)$. 
\end{prop}

\begin{remark}
As we will see below, it is essential in the proof of Theorem \ref{thm_B}  that   the orbit of a point in a Lebesgue full measure set  hits the support of $\mu_j$ in a \emph{finite time}. 
This is contrastive to the example in Section \ref{S_NCP} in which the orbit of any point in a positive Lebesgue measure set does not hit  the support of any ergodic invariant probability measure in a finite time (although the $\omega$-limit set of the orbit includes the support of an ergodic invariant probability measure). 
\end{remark}

We will prove Proposition \ref{prop:0317} under the help of previous 
work by Tsujii \cite{tj_C}
 in a functional theoretic framework.
We first recall that  the \emph{Perron--Frobenius operator}  $\mathcal L: L^1(\mathrm{Leb}) \longrightarrow L^1(\mathrm{Leb})$ of 
 $F$ is given by
\[
\mathcal L \varphi (x) =\sum _{F(y)=x} \frac{ \varphi (y)}{\vert \det DF(y) \vert} = \sum _{i=1}^p  \frac{1_{D_i}\cdot \varphi }{\vert \det DF \vert} \circ F_i^{-1}(x)
\]
for $\varphi \in L^1 (\mathrm{Leb})$, $x\in D$. 
It is straightforward to see that
\begin{equation}\label{eq:duality}
\int _D \psi \circ F^n \cdot \varphi \,d\mathrm{Leb} =\int _D \psi  \cdot \mathcal L^n \varphi \,d\mathrm{Leb} 
\end{equation}
for $\psi\in L^\infty (\mathrm{Leb})$, $\varphi \in L^1(\mathrm{Leb})$ and $n\geq 1$, 
so that several statistical properties of $f$ are reduced to  nice spectral properties of $\mathcal L$ on a good Banach space.
A standard reference for Perron--Frobenius operators 
 is  the monographs of Baladi \cite{ba1,ba2}, and here we just summarize  known  result related to our purpose.

We recall that the Perron--Frobenius operator is said to satisfy the \emph{Lasota--Yorke inequality} on a Banach space $E\subset L^1(\mathrm{Leb})$ equipped with the norm $\Vert \cdot \Vert$ if  there exist  a positive integer $n_0$ and real numbers $0<\alpha <1$, $ \beta >0$   such that $\Vert \varphi\Vert _{L^1} \leq \Vert \varphi\Vert$ for each $\varphi \in E$,  the inclusion $(E, \Vert  \cdot  \Vert) \hookrightarrow (L^1(\mathrm{Leb}), \Vert \cdot \Vert _{L^1})$ is compact, $\mathcal L$ can be extended to a bounded operator on $E$  and 
\begin{equation}\label{eq:forProp35}
\Vert \mathcal L^{n_0}\varphi \Vert \leq \alpha \Vert \varphi \Vert + \beta  \Vert \varphi \Vert _{L^1} \quad \text{for each $\varphi \in E$}.
\end{equation}
It is well known as the Ionescu-Tulcea--Marinescu and Hennion theorem (\cite{he, hh}) that if $\mathcal L$ satisfies the Lasota--Yorke inequality on $E$, then the essential spectral radius of $\mathcal L$ is 
bounded by $\alpha$.

It also follows from the Lasota--Yorke inequality  that the spectral radius of $\mathcal L: E\to E$ is $1$  (so that $\mathcal L$ is quasi-compact)  by a standard argument, as follows: 
$\Vert \mathcal L \varphi\Vert _{L^1} \leq \Vert \varphi \Vert _{L^1}$ for all $\varphi \in L^1(\mathrm{Leb})$ due to \eqref{eq:duality}.
Hence,  by using \eqref{eq:forProp35} repeatedly, we get that for each $n=n_0k +l$ with $k, l \in \mathbb N_0$,  $0\leq l \leq n_0-1$, 
\[
\Vert \mathcal L^{n}\varphi \Vert \leq C_1^l \alpha ^k \Vert \varphi \Vert + C_1^l\beta (1+\alpha + \cdot + \alpha ^{k-1}) \Vert \varphi \Vert _{L^1} \leq C_2\Vert \varphi \Vert 
\]
where $C_1=\max\{ 1, \Vert \mathcal L\Vert\}$ and $C_2=C_1^{n_0}(1+\beta (1-\alpha)^{-1})$, 
and thus  the spectral radius of $\mathcal L:E \to E$ is bounded by $1$. 
 Furthermore, $1$ is in the spectrum of $\mathcal L:E \to E$ 
 because
 $1_D$ is the eigenfunction with eigenvalue $1$ of $\psi \mapsto \psi \circ F$ on $L^\infty(\mathrm{Leb})$, so that $1$ is an eigenvalue of the adjoint operator $\mathcal L^*$ of $\mathcal L$ on $ E^* \supset(L^1(\mathrm{Leb}))^* \cong L^\infty(\mathrm{Leb})$ due to  \eqref{eq:duality}.

On the other hand, 
for real analytic planar piecewise expanding maps,  
Buzzi and Tsujii showed independently the Lasota--Yorke inequality (where $E$ is   the quasi-Lipschitz space $\mathfrak{B}\mathfrak{V}$ in \cite[Proposition 1.4]{bu}  and the bounded variation space $BV$ in \cite[Proposition 10]{tj_C}, respectively). 
We also see \cite{th,ba2} and references therein for several generalizations of their results.

Under the help of the above observations, one can obtain the following lemma.
\begin{lemma}\label{lem:0519c}
Assume that $\mathcal L$ satisfies the Lasota--Yorke inequality on a Banach space $E\subset  L^1(\mathrm{Leb})$. 
Furthermore, we assume that 
\begin{itemize}
\item[(D)]
$E\cap L^\infty(\mathrm{Leb})$ is dense in $E$. 
\end{itemize}
Then, there exist  $k, q\in \mathbb N$ and density functions $h_1,\ldots ,h_q\in E$ with 
mutually disjoint supports such that  $1$ is the only eigenvalue of $\mathcal L^k$ with absolute value $1$ with multiplicity  $q$ and  the eigenspace $E_1$ of $\mathcal L^k$ associated to the eigenvalue $1$ is spanned by $h_1,\ldots ,h_q$.
\end{lemma}
\begin{remark}
Baladi and Gou\"ezel showed in \cite[Theorem 33]{bg}   exponential decay of   correlation functions 
 for a large class of piecewise hyperbolic maps  under the assumptions including that the condition (D).
By modifying their argument, Thomine showed in \cite[Section 6]{th}    exponential decay of   correlation functions for a large class of  piecewise expanding maps. 
Both works proved their versions of Lemma \ref{lem:0519c}  (for general $E$ which satisfies certain conditions including the condition (D) in \cite{bg} and for Sobolev spaces $E$ in \cite{th}), and moreover they used the versions to show their exponential decay of correlation functions.
In fact, Lemma \ref{lem:0519c} can be proven  by literally repeating their arguments, so below we just give an outline of the proof.

The advantage of the functional space $\mathfrak{B}\mathfrak{D}$ is the fact that $\mathfrak{B}\mathfrak{D} \subset L^\infty(\mathrm{Leb}) \cap BV$, so that  the condition (D) is satisfied for $\mathfrak{B}\mathfrak{D}$, 
whereas $BV\not\subset L^\infty(\mathrm{Leb})$ 
 (refer to \cite[Lemma 1]{bk}).
 Hence, we shall only   consider the case when $E=\mathfrak{B}\mathfrak{V}$.\footnote{It seems that Thomine  used a version of Lemma \ref{lem:0519c}  in   the case  $E=BV$   in the proof of  \cite[Theorem  2.6]{th}, but we have no idea whether the condition (D) holds for $BV$, 
so we will not apply   Lemma \ref{lem:0519c} to $E=BV$.}
 \end{remark}

 \begin{proof}[Outline of the proof of Lemma \ref{lem:0519c}]
As observed above, since $\mathcal L$ satisfies the Lasota--Yorke inequality on a Banach space $E\subset L^1(\mathrm{Leb})$, 
the essential spectral radius $\rho _{\mathrm{ess}}(\mathcal L)$ of $\mathcal L$ is strictly smaller than $1$ and the spectral radius of $\mathcal L$ is $1$. 
Therefore, the spectrum of $\mathcal L$ outside of the ball of radius $\rho _{\mathrm{ess}}(\mathcal L)$ consists of finitely many eigenvalues with finite multiplicities.
 Furthermore, by repeating the argument in ``First step'' of \cite[Section 6]{th} (using the condition  (D)), one can see that the eigenvalues of $\mathcal L$ with absolute value $1$ is a cyclic group and have  no nontrivial Jordan block.
Hence, 
  there is an integer   $k\in \mathbb N$ such that 
   $1$ is the only eigenvalue of $\mathcal L^k$ with absolute value $1$ and the multiplicity of $1$ is finite (denote it by $q$).
  
Let  $\pi_1:E\longrightarrow E_1$ be the projection to   the eigenspace $E_1$  associated to the eigenvalue $1$   and $\mu:= \pi _1(1_D)d\mathrm{Leb}$.   
By repeating the former part of the argument in ``Second step'' of \cite[Section 6]{th}, one can show that there are mutually disjoint $F^k$-invariant sets $B_1, \ldots , B_q$ such that $\mu_j:=h_j d\mathrm{Leb}$ are absolutely continuous ergodic $F^k$-invariant probability measures, where $h_j:=\frac{1_{B_j} \pi _1(1_D)}{\mu (B_j)}$, 
and $E_1$ is spanned by $h_1, \ldots , h_q$ as required.
 \end{proof}

\begin{lemma}\label{lem:0519d}
Let $k$ and $h_1,\ldots ,h_q$ be as in Lemma \ref{lem:0519c}. 
Then, 
 there exist positive linear functionals 
  $\lambda _1, \ldots,\lambda _q $  on $E$ such that
 \begin{align}\label{eq:0519a}
 \lim_{n\to\infty}\left\Vert \mathcal L^{kn}\left(\varphi -\sum_{j=1}^q \lambda _j(\varphi )h_j \right)\right\Vert =0
\end{align}
for every $\varphi \in E$.
\end{lemma}
\begin{proof}
By construction, the spectral radius   of $\mathcal L^{k} -\pi _1$   is strictly smaller than $1$, so that one can find   constants $C>0$ and $0<\rho <1$ such that
\[
\Vert \mathcal L^{kn} \varphi -\pi _1 (\varphi ) \Vert \leq C\rho ^n\Vert \varphi \Vert
\]
for any $\varphi \in E$ and $n\geq 1$.
Since $E_1=\pi _1(E)$ is spanned by $h_1,\ldots, h_q$ for each $\varphi \in E$, one can find real numbers $\lambda _1(\varphi ), \ldots , \lambda _q(\varphi )$ such that $\pi_1(\varphi )=\sum _{j=1}^q \lambda _j(\varphi ) h_j$, and get the desired convergence. 

It remains to show that the $\lambda_j$ are positive linear functionals, but it can be done by a standard argument (cf.~\cite{ry}). 
Denote $\varphi -\sum_{j=1}^q \lambda _j(\varphi )h_j $ by $\widetilde\pi_1(\varphi )$.
For $\varphi, \psi \in E$, 
since $\varphi +\psi \in E$ and $h_j$ is $\mathcal L^k$-invariant, we have
\begin{align*}
0&= \lim _{n\to\infty} \mathcal L^{kn}\widetilde\pi_1\left(\varphi +\psi\right)\\
&= \lim _{n\to\infty} \Big(\mathcal L^{kn}\widetilde\pi_1\left(\varphi  \right)+
\mathcal L^{kn}\widetilde\pi_1\left( \psi\right)
+
\mathcal L^{kn}\Big(\sum _{j=1}^q (\lambda _j(\varphi +\psi) - \lambda _j(\varphi ) -\lambda _j(\psi)) h_j\Big)\Big)\\
&= \sum _{j=1}^q (\lambda _j(\varphi +\psi) - \lambda _j(\varphi ) -\lambda _j(\psi)) h_j 
\end{align*}
Therefore, since $h_1, \ldots ,h_q$ are linearly independent, we get the additivity of $\lambda _j$ for each $1\leq j\leq q$. 
The positivity  and homogeneity of $\lambda _j$ immediately follow from the linearity and positivity of $\mathcal L^k$.
\end{proof}

\begin{remark}\label{rmk:0519} 
One may be able to apply  Bartoszek's asymptotic periodicity theorem for abstract quasi-compact positive operators on Banach lattices (\cite[Theorem 1, Proposition 1 and Lemma 5]{br}).
In fact, it was shown that, if  $E\subset L^1(\mathrm{Leb})$ is a Banach lattice (not necessarily satisfies the condition (D)), the spectral radius of $\mathcal L$ is $1$ and the essential spectral radius of $\mathcal L$ is strictly smaller than $1$, then
 there exist finitely many linearly independent density functions  $h_1',   \ldots , h_{q'}'\in E$, positive functionals $\lambda _1', \ldots,\lambda _{q'}'\in E^*$ and a permutation $\alpha$ such that $\mathcal Lh_j'= h_{\alpha(j)}'$ and 
\[
 \lim_{n\to\infty}\left\Vert \mathcal L^n\left(\varphi -\sum_{j=1}^{q'} \lambda _j'(\varphi )h_j '\right)\right\Vert =0
\]
for every $\varphi \in E$.
This immediately implies the key estimate \eqref{eq:0519a}.
However, we have no idea whether $BV$ or $\mathfrak{B}\mathfrak{D}$ is a Banach lattice.
\end{remark}

\begin{proof}[Proof of Proposition \ref{prop:0317}]
Let $E=\mathfrak{B}\mathfrak{D}$, then by the previously-mentioned Buzzi's result and the remark above, $\mathcal L$ satisfies the Lasota--Yorke inequality on $E$ and  the condition (D) holds for $E$, so that  Lemmas \ref{lem:0519c} and \ref{lem:0519d} are applicable for the proof.
Let $k$, $h_1, \ldots , h_q$ and $\lambda _1,\ldots ,\lambda _q$ be as in Lemma \ref{lem:0519c}.
Let $\mu _j$ denote the probability measure whose density function is $h_j$  for $j=1,\ldots ,q$. 
Let $A:=D\setminus  \bigcup_{j=1}^q \mathrm{supp} (\mu_j)$, so that $\int _D 1_A \cdot h_j \, d\mathrm{Leb}=0$ for each $j=1,\ldots ,q$.
Let $B$ be the set of points ${x}\in D$ such that $F^n({x})\ \in A$ for all $n\in \mathbb N$. 
It suffices to show that $\mathrm{Leb}(B) =0$.

It follows from \eqref{eq:duality} and \eqref{eq:0519a}  that, for any $\psi\in L^\infty(\mathrm{Leb})$,
 \begin{align*}
&\left\vert \int _D\psi \circ F^{nk}\cdot  \varphi \, d\mathrm{Leb} -\sum_{j=1}^q \lambda _j(\varphi ) \int _D\psi \cdot h_{j} \, d\mathrm{Leb} \right\vert\\
&\hspace{50pt}= \left\vert \int _D\psi \cdot   \mathcal L^{nk}\left(\varphi -\sum_{j=1}^q \lambda _j(\varphi )h_j \right) d\mathrm{Leb} \right\vert 
\to 0 \quad (n\to\infty).
 \end{align*}
 In particular we get
 \[
\lim _{n\to\infty} \int _D 1_A \circ F^{nk}  d\mathrm{Leb} =0.
 \]
 On the other hand,
by definition of $B$, for any $n\in \mathbb N$,
\begin{align*}
\mathrm{Leb}(B) = \int _B 1_D\,d\mathrm{Leb}   = \int _B 1_{A} \circ F^{nk}\,d\mathrm{Leb} 
\leq \int _D 1_{A} \circ F^{nk}\,d\mathrm{Leb}.
\end{align*}
Combining these estimates, we get that $\mathrm{Leb}(B)=0$. 
This completes the proof of Proposition \ref{prop:0317}.
\end{proof}

Now we are ready to prove Theorem \ref{thm_B}.

\begin{proof}[Proof of Theorem \ref{thm_B}]
Since  $f$ can be extended to a differential map on a  compact set $\overline D_i$ and $D$ is the union of such compact sets,  the function ${x}\mapsto \max\left\{ \log \left\lVert DF({x}) \right\rVert , 0\right\}$  is bounded, in particular, $\mu_j$-integrable for each $j=1,\dots, q$.
Hence, we can apply  the Oseledets multiplicative ergodic theorem to the cocycle $(DF^{n})_{n\in \mathbb N}$ over the ergodic measure-preserving system $F$ on $(D, \mu_j)$ for each $j=1,\dots,q$: 
There exists  a $\mu_j$-full measure set $X_j$  such that 
 \begin{equation}\label{eq:0319a}
\lim_{n\to\infty}\frac{1}{n}\log \left\lVert DF^n({x}){\boldsymbol{v}}\right\rVert\quad \text{exists for all }({x},{\boldsymbol{v}})\in X_j \times (\mathbb{R}^2\backslash \{\mathbf{0}\}).
\end{equation}
 Let $V_j:=\bigcup _{N=0}^\infty F^{-N}(\mathrm{supp}(\mu_j))$. 
 Then, by Proposition \ref{prop:0317}, 
 \begin{equation}\label{eq:0319}
 \mathrm{Leb}\biggl(D\setminus \bigcup _{j=1}^q V_j\biggr) =0.
 \end{equation}
 Note that for each measurable set $A\subset \mathrm{supp} (\mu_j)$, if  $\mathrm{Leb}(A) >0$, then  one can find 
a non-negative integer $n$ such that $\mathrm{Leb}(A_n) >0$, where $A_n:=\{  x \in A\mid \frac{1}{n+1} \leq h_j(x) < \frac{1}{n} \} $ 
for $n=1,2,\dots$, $A_0:=\{ {x} \in A\mid h_j({x}) \geq 1 \} $ and $h_j$ is the Radon-Nikodym derivative  of $\mu_j$. 
 Thus  we get
 $\mu_j(A) \geq \mathrm{Leb}(A_n)/(n+1) >0$.  
 By applying the contraposition of this implication to the $\mu_j$-zero measure set $\mathrm{supp}(\mu_j) \setminus X_j$,  we get  $\mathrm{Leb}(\mathrm{supp}(\mu_j) \setminus X_j)=0$.
Hence 
\begin{equation*}
\mathrm{Leb}\biggl(V_j  \setminus \bigcup_{N=0}^\infty F^{-N}(X_j) \biggr) =0.
\end{equation*}
By this fact together with \eqref{eq:0319}, we conclude that for $\mathrm{Leb}$-almost every ${x}\in D$, 
 \begin{equation}\label{eq:0319b}
F^N({x}) \in X_j \quad \text{for some integer $j\in \{1,\dots,q\}$ and $N\geq 0$.}
\end{equation}
Let ${x}'=F^N({x})$ and ${\boldsymbol{v}}'= DF^N({x}) {\boldsymbol{v}}$ for any element ${\boldsymbol{v}}$ of $T_{{x}} D\setminus \{\mathbf{0}\}$.
Then we observe that
\begin{align*}
&\left\vert\, \frac{1}{n} \log \left\Vert DF^n({x}){\boldsymbol{v}}\right\Vert  - \frac{1}{n} \log \left\Vert DF^n({x}'){\boldsymbol{v}}'\right\Vert\, \right\vert\\
&\hspace{50pt}\leq \frac{2N}{n} \max\biggl\{\log \left\Vert DF_i({y})\right\Vert\,\Big|\, {y}\in \overline D_i, 
i=1,\dots,p\biggr\} \to 0\quad\text{as $n\to\infty$}.
\end{align*}
This observation together with \eqref{eq:0319a} and \eqref{eq:0319b} completes the proof.
\end{proof}

\subsection*{Acknowledgements.}
The authors would like to thank Shin Kiriki and Masato Tsujii for helpful comments and suggestions.
This work was partially supported by JSPS KAKENHI Grant Number 19K14575, 22K03342 
and WISE program (MEXT) at Kyushu University.

\end{document}